\documentclass[12pt]{amsart}

\usepackage{amstext}
\usepackage{amsthm}
\usepackage{amsmath}
\usepackage{amssymb}
\usepackage{latexsym}
\usepackage{amsfonts}
\usepackage{graphicx}
\usepackage{texdraw}
\usepackage{graphpap}

\usepackage[pagebackref,hypertexnames=false, colorlinks, citecolor=red, linkcolor=red]{hyperref} 
\usepackage[backrefs]{amsrefs}

\input txdtools

\begin{document}


\title[Thin Sequences, $H^p$ Theory, Model Spaces, and Uniform Algebras]{Thin Sequences and Their Role in $H^p$ Theory, Model Spaces, and Uniform Algebras}


\author[P. Gorkin]{Pamela Gorkin$^\dagger$}
\address{Pamela Gorkin, Department of Mathematics\\ Bucknell University\\  Lewisburg, PA  USA 17837}
\email{pgorkin@bucknell.edu}
\thanks{$\dagger$ To Lund University, Mathematics in the Faculty of Science for financial support and hospitality provided during the summer of 2012}

\author[S. Pott]{Sandra Pott}
\address{Sandra Pott\\ Faculty of Science\\
Centre for Mathematical Sciences\\
Lund University\\
22100 Lund, Sweden}
\email{sandra@maths.lth.se}

\author[B.D.Wick]{Brett D. Wick$^\ddagger$}
\address{Brett D. Wick, School of Mathematics\\ Georgia Institute of Technology\\ 686 Cherry Street\\ Atlanta, GA USA 30332-0160}
\email{wick@math.gatech.edu}
\thanks{$\ddagger$ Research supported in part by a National Science Foundation DMS grants \# 1001098 and \# 0955432.}





%
%

\newcommand{\ci}[1]{_{ {}_{\scriptstyle #1}}}

\newcommand{\norm}[1]{\ensuremath{\left\|#1\right\|}}
\newcommand{\abs}[1]{\ensuremath{\left\vert#1\right\vert}}
\newcommand{\p}{\ensuremath{\partial}}
\newcommand{\pr}{\mathcal{P}}

\newcommand{\pbar}{\ensuremath{\bar{\partial}}}
\newcommand{\db}{\overline\partial}
\newcommand{\D}{\mathbb{D}}
\newcommand{\B}{\mathbb{B}}
\newcommand{\Sp}{\mathbb{S}}
\newcommand{\T}{\mathbb{T}}
\newcommand{\R}{\mathbb{R}}
\newcommand{\Z}{\mathbb{Z}}
\newcommand{\C}{\mathbb{C}}
\newcommand{\N}{\mathbb{N}}
\newcommand{\Hc}{\mathcal{H}}
\newcommand{\scrL}{\mathcal{L}}
\newcommand{\td}{\widetilde\Delta}

\newcommand{\CC}{\mathcal{C}}
\newcommand{\EE}{\mathcal{E}}
\newcommand{\RR}{\mathcal{R}}
\newcommand{\TT}{\mathcal{T}}

\newcommand{\La}{\langle }
\newcommand{\Ra}{\rangle }
\newcommand{\rk}{\operatorname{rk}}
\newcommand{\card}{\operatorname{card}}
\newcommand{\ran}{\operatorname{Ran}}
\newcommand{\osc}{\operatorname{OSC}}
\newcommand{\im}{\operatorname{Im}}
\newcommand{\re}{\operatorname{Re}}
\newcommand{\tr}{\operatorname{tr}}
\newcommand{\vf}{\varphi}

\renewcommand{\qedsymbol}{$\Box$}
\newtheorem{thm}{Theorem}[section]
\newtheorem{lm}[thm]{Lemma}
\newtheorem{cor}[thm]{Corollary}
\newtheorem{conj}[thm]{Conjecture}
\newtheorem{prob}[thm]{Problem}
\newtheorem{prop}[thm]{Proposition}
\newtheorem*{prop*}{Proposition}
\newtheorem{defin}[thm]{Definition}
\theoremstyle{remark}
\newtheorem{rem}[thm]{Remark}
\newtheorem*{rem*}{Remark}


%
%

\begin{abstract}
In this paper we revisit some facts about thin interpolating sequences in the unit disc from three perspectives: uniform algebras, model spaces, and $H^p$ spaces. We extend the notion of asymptotic interpolation to $H^p$ spaces, for $1 \le p \le \infty$, providing several new ways to think about these sequences.
\end{abstract}

%
%

\maketitle

\section{Introduction and Motivation}

Let $\{z_j\}$ be a sequence of points in $\mathbb{D}$. We say that $Z:=\{z_j\}_{j=1}^\infty$ is an \textnormal{interpolating sequence} for $H^\infty$, the space of bounded analytic functions, if for every $w\in\ell^\infty$ there is a function $f\in H^\infty$ that solves the interpolation problem $$f(z_j) = w_j, ~\mbox{where}~ j\in\N.$$ Carleson's interpolation theorem says that $\{z_j\}_{j=1}^\infty$ is an interpolating sequence for $H^\infty$ if and only if 
\begin{equation}
\label{Interp_Cond}
\delta = \inf_{j}\delta_j:=\inf_j \left\vert B_j(z_j)\right\vert=\inf_{j}\prod_{k \ne j} \left|\frac{z_j - z_k}{1 - \overline{z_j} z_k}\right| > 0.
\end{equation}
Here we are letting
$$
B_j(z)=\prod_{k\neq j}\frac{-\overline{z_k}}{\abs{z_k}}\frac{z-z_k}{1-\overline{z_k}z}
$$
denote the infinite Blaschke factor that vanishes on the set of points $Z\setminus\{z_j\}=\{z_k:k\neq j\}$.
It is known, see page 285 in \cite{Garnett}, that given an interpolating sequence, if we let 
\begin{equation}
M = \sup_{\|a\|_{\ell^\infty} \le 1} \inf\left\{\|f\|_\infty: f \in H^\infty, f(z_j) = a_j, j\in\N\right\}
\end{equation}
denote the constant of interpolation, then there are functions $f_j \in H^\infty$ such that 
\begin{equation}   \label{eq:beurprop}
f_j(z_j) = 1, f_j(z_k) = 0~\mbox{for}~k \ne j~\mbox{and}~\sup_{z\in\mathbb{D}}\sum_j |f_j(z)| \le M.
\end{equation}
The $f_j$ are often referred to as P. Beurling functions (see \cite{J1}), and they are not explicitly defined. However, Peter Jones \cite{J} found a simple formula for
functions in $H^\infty$ with properties close to those of P. Beurling functions, in the sense that the inequality (\ref{eq:beurprop}) is satisfied with a constant $M' \ge M$.

As a result, Jones obtained a simpler proof of Carleson's interpolation theorem. Interpolating sequences were studied by Shapiro and Shields \cite{SS} who showed that for $1 \le p \le \infty$, Carleson's condition is also a necessary and sufficient condition for interpolation in $H^p$ in the following sense: for each $\{a_j\} \in \ell^p$ there exists $f \in H^p$ with $f(z_j)(1 - |z_j|^2)^{1/p} = a_j$ for all $j$.



In this paper, we will be interested in sequences $Z$ that satisfy a stronger condition than \eqref{Interp_Cond}.  Recall that a sequence $Z=\{z_j\}\subset\D$ is \textnormal{thin} if 
$$
\lim_{j\to\infty}\delta_j:=\lim_{j\to\infty}\prod_{k\neq j}\left\vert\frac{z_j-z_k}{1-\overline{z_k}z_j}\right\vert=1.
$$

Thin sequences have played an important role in function theory on the unit disc, and as motivation for many of our main results we highlight some of the interesting connections thin sequences have to function theory and functional analysis.

First, we recall that thin sequences are connected to certain functional analytic basis properties of normalized reproducing kernels.  Let $\Hc$ be a complex Hilbert space. Recall that a sequence $\{x_n\}$ in $\Hc$ is said to be {\it complete} if $\mbox{span} \{x_n: n \ge 1\} = \Hc$,  {\it minimal} if for every $n \ge 1$, it is the case that $x_n \notin ~\mbox{span}\{x_m: m \ne n\}$ and {\it Riesz} if there are constants $C_1$ and $C_2$, positive, such that for all complex sequences $\{a_n\}$ we have
$$C_1 \sum_{n \ge 1} |a_n|^2 \le \left\|\sum_{n \ge 1} a_n x_n\right\|^2_{\Hc} \le C_2 \sum_{n \ge 1} |a_n|^2.$$ Recall that every Riesz sequence is minimal, but the converse is not necessarily true. Finally, the Gram matrix corresponding to $\{x_j\}$ is the matrix $G = \left(\langle x_n, x_m \rangle\right)_{n, m \ge 1}$.  

In \cite{CFT}, asymptotically orthonormal sequences (AOS) and asymptotically orthonormal basic sequences (AOB) were studied; that is, a sequence $\{x_n\}$ is an AOS in $\Hc$ if there exists $N_0 \in \N$ such that for all $N \ge N_0$ there exist positive constants $c_N$ and $C_N$ with

\begin{eqnarray}
\label{thininequality} 
c_N \sum_{n \ge N} |a_n|^2 \le \left\|\sum_{n \ge N} a_n x_n\right\|^2_{\Hc} \le C_N \sum_{n \ge N} |a_n|^2,
\end{eqnarray}
and $c_N \to 1, C_N \to 1$ as $N \to \infty$. If we can take $N_0 = 1$, the sequence is said to be an AOB; this is equivalent to being AOS and a Riesz sequence.  In Proposition 3.2 of \cite{CFT}, Chalendar, Fricain and Timotin note that work of Volberg, Theorem 3 in \cite{V}, implies the following.

\begin{prop}\label{propCFT} Let $\{x_n\}$ be a sequence in $\Hc$. The following are equivalent:
\begin{enumerate}
\item $\{x_n\}$ is an AOB;
\item There exist a separable Hilbert space $\mathcal{K}$, an orthonormal basis $\{e_n\}$ for $\mathcal{K}$ and $U, K: \mathcal{K} \to \mathcal{H}$, $U$ unitary, $K$ compact, $U + K$ left invertible, such that 
$$(U + K)(e_n) = x_n;$$
\item The Gram matrix $G$ associated to $\{x_n\}$ defines a bounded operator of the form $I + K$ with $K$ compact.
\end{enumerate} 
\end{prop}

We now make the connection to thin sequences explicit.  Given an inner function $\Theta$, we define the corresponding model space to be $K_\Theta = H^2 \ominus \Theta H^2$ and the orthogonal projection will be denoted by $P_\Theta$. The reproducing kernel in $K_\Theta$ for $z_0 \in \mathbb{D}$ is 
$$
k_{z_0}^\Theta(z) = \frac{1 - \overline{\Theta(z_0)}{\Theta(z)}}{1 - \overline{z_0}z}
$$ 
and the normalized reproducing kernel is 
$$
h_\lambda^\Theta(z) = \sqrt{\frac{1 - |z_0|^2}{1 - |\Theta(z_0)|^2}} k_{z_0}^\Theta(z).
$$
Finally, note that $$k_{z_0} = k_{z_0}^\Theta + \Theta \overline{\Theta(z_0)}k_{z_0}.$$

It is well known that if $\{z_n\}$ is a Blaschke sequence with simple zeros with corresponding Blaschke product $B$, then $h_{z_n}=\frac{(1-\left\vert z_n\right\vert^2)^{\frac{1}{2}}}{(1-\overline{z_n}z)}$ is a complete minimal system in $K_B$ and we also know that $\{z_n\}$ is interpolating if and only if $\{h_{z_n}\}$ is a Riesz basis. 

The following beautiful theorem provides the connection back to thin sequences.
\begin{thm}[Volberg, Theorem 2 in \cite{V}] 
\label{Volberg} 
The following are equivalent:
\begin{enumerate}
\item$\{z_n\}$ is a thin interpolating sequence;
\item
The sequence $\{h_{z_n}\}$ is a complete $AOB$ in $K_B$;
\item
There exist a separable Hilbert space $\mathcal{K}$, an orthonormal basis $\{e_n\}$ for $\mathcal{K}$ and $U, K: \mathcal{K} \to K_B$, $U$ unitary, $K$ compact, $U + K$ invertible, such that 
$$(U + K)(e_n) = h_{z_n} \text{ for all } n \in \N.$$
\end{enumerate}
 \end{thm}

We also have the following very useful equivalent conditions.
\begin{prop}[Chalendar, Fricain, Timotin, Proposition 4.1 in \cite{CFT}]
\label{prop2CFT} 
If $\{z_n\}$ is a Blaschke sequence of distinct points in $\mathbb{D}$, then the following are equivalent:
\begin{enumerate}
\item $\{h_{z_n}\}$ is a complete $AOB$ in $K_B$;
\item $(G - I)e_n \to 0$.
\end{enumerate}
\end{prop}

Finally, as further motivation, we recall the connection between thin sequences and a special interpolation problem.  Recall that a sequence $Z=\{z_j\}\subset\D$ is \textnormal{asymptotically interpolating of type 1} if for any sequence $w\in \mbox{ball}(\ell^\infty)$ there exists a function $f\in H^\infty$ with $\norm{f}_\infty\leq 1$ such that
$$
\abs{f(z_n)-w_n}\to 0.
$$

The following theorem shows that thin sequences are precisely the sequences for which interpolation can be done in $QA=H^\infty\cap VMO$ and they are, therefore, eventually interpolating sequences for $H^\infty$.  Here $VMO$ is the space of functions on the unit circle with vanishing mean oscillation. 

\begin{thm}[Wolff, Theorem III.10 in \cite{WT}] The following conditions on a sequence $\{z_n\}$ of distinct
points in $\D$ are equivalent:
\begin{enumerate}
\item If $a \in\ell^\infty$, then there is $f\in QA$ with $f(z_n)=a_n$ for all $n$;
\item If $a\in\ell^\infty$ and $\varepsilon > 0$, then there is $f\in H^\infty$ with
$\| f\|_{\infty} <\limsup|a_n|+\varepsilon$ and $f(z_n)=a_n$ for all but finitely many $n$;
\item $\{z_n\}$ is a thin interpolating sequence.
\end{enumerate}
\end{thm}

The second equivalent condition indicates that we can interpolate a sequence $\{a_j\}$ of norm $1$ with a function of norm at most $1 + \varepsilon$ if we ignore finitely many terms. The number of terms appears to depend on the sequence $\{a_j\}$. One of our key observations is that, in fact, the number of terms we can ignore depends only on $\varepsilon$ and not the sequence $\{a_j\}$, see Theorem \ref{main} below.

Asymptotic interpolating sequences (AIS), introduced in \cite{HKZ}, were originally a tool to study essential norms of composition operators. However, further study of asymptotic interpolating sequences \cite{GM} showed that they were closely related to thin sequences. In fact, a sequence $\{z_n\}$ is thin if and only if it is an asymptotic interpolating sequence of type $1$; thus, asymptotic interpolating sequences of type $1$ are those for which the norm is the best one can hope for. In this paper, we study thin sequences from various angles: We consider them from a uniform algebra perspective reminiscent of work of Wolff \cite{W}, an $H^p$ perspective in the spirit of Shapiro and Shields, and from the point of view of model spaces as in Volberg \cite{V} and later in Chalendar, Fricain, and Timotin \cite{CFT}.  We discuss these in more detail below.

\subsection{Main Results and Structure of the Paper}

In Section~\ref{P. Beurling} we use the Commutant Lifting Theorem and P. Beurling functions to provide new proofs of a characterization of thin sequences due to Dyakonov and Nicolau to obtain other characterizations of thin sequences.   In \cite{DN} Dyakonov and Nicolau studied interpolation by nonvanishing functions and used their result to provide a proof that every thin sequence is an asymptotic interpolating sequence of type $1$ that does not rely on deep maximal ideal space techniques. In Section~\ref{asiinfty}, we will present another proof of the converse, namely that an asymptotic interpolating sequence of type $1$ is thin, that does not depend on knowledge of support sets for points in the maximal ideal space. It does, however, depend on a result of Carleson and Garnett (see, for example, Theorem 4.1 in \cite{Garnett}) or, in its place, the fact that an asymptotically interpolating sequence is eventually an interpolating sequence for $H^\infty$, which can be found in Theorem 1.6 in \cite{GM} . 

Our next result involves the algebra $H^\infty + C$: Let $C$ denote the space of continuous functions on the unit circle and, thinking of $H^\infty$ as an algebra on the circle, we let $H^\infty + C = \{f + g: f \in H^\infty, g \in C\}$. Sarason \cite{Sarason1} showed that $H^\infty + C$ is a closed subalgebra of $L^\infty$. We let $$QA = \overline{H^\infty + C} \cap H^\infty,$$ where the bar denotes complex conjugation. We are interested in a result, due to Axler and the first author \cite{AG} and proved independently by Guillory, Izuchi and Sarason \cite{GIS}, that says that given an $H^\infty$ function $f$ that tends to zero on an interpolating sequence $\{z_n\}$, if $B$ is the corresponding interpolating Blaschke product, then $f \overline{B} \in H^\infty + C$. In this section, we also provide another proof of this multiplication using the Jones construction, \cite{J}. Finally, as a consequence, we show how a result of Wolff \cite{W} for multiplication by functions in $QA$ tending asymptotically to zero follows from this.

In Section~\ref{asip}, we extend the notion of asymptotic interpolation to $H^p$ spaces, $1 \le p \le \infty$, and provide several new equivalent definitions of asymptotic interpolation for $H^p$ for a sequence $Z = \{z_n\}$ and we show that all are equivalent to $Z$ being thin. This result, which is the main result of this section, is our Theorem~\ref{main}. 

\subsection{Notation}
\label{preliminaries}
The following notation will be standard throughout the paper.  As usual, for $1 \le p < \infty$, $H^p$ denotes the Hardy space on the open unit disc $\mathbb{D}$ and $H^\infty$ denotes the algebra of bounded analytic functions on $\mathbb{D}$.  The norm of a function $f\in H^p$ will be denoted by $\left\| f\right\|_p$.  For $f \in H^\infty$ we let $f^\star$ denote the radial limit function of $f$. Identifying $f$ with $f^\star$ allows us to think of $H^\infty$ as a subalgebra of the algebra of essentially bounded measurable functions on the unit circle, $\mathbb{T}$. Letting $C$ denote the algebra of continuous functions on the unit circle, we let $H^\infty + C$ denote the subalgebra of $L^\infty$ (see \cite{Sarason1}) consisting of functions $g$ of the form $g = h + c$, where $h \in H^\infty$ and $c \in C$. Two more algebras will play an important role in this paper: the algebra $QC:= (H^\infty + C) \cap \overline{(H^\infty + C)}$, where the bar denotes complex conjugation, and the algebra $QA = QC \cap H^\infty$.

We will let $\ell^\infty$ denote the collection of sequences $\{a_k\}$ such that
$$
\norm{a}_{\ell^\infty}:=\sup_{n}\abs{a_n}<\infty,
$$ and $\ell^p$ the collection of sequences $\{a_k\}$ such that
$$\norm{a}_{\ell^p}:=\left(\sum_{n = 1}^\infty |a_n|^p\right)^{1/p}<\infty.$$
For an integer $N$, we let $\|a\|_{N, \ell^p} := \|\{a_j\}_{j \ge N}\|_{\ell^p}$.

Finally, we will require some basic information about the maximal ideal space of $H^\infty$. We recall everything we need here.  The space of nonzero multiplicative linear functionals on $H^\infty$ is called the maximal ideal space of $H^\infty$ and we denote it by $M(H^\infty)$. We note that by identifying a point of $\mathbb{D}$ with point evaluation, we may think of $\mathbb{D}$ as a subset of $M(H^\infty)$. For $H^\infty + C$ it is well known that  $M(H^\infty + C) = M(H^\infty) \setminus \mathbb{D}$. Carleson's Corona Theorem tells us that $\mathbb{D}$ is dense in $M(H^\infty)$. 
Moreover, $M(L^\infty)$ is naturally embedded into $M(H^\infty)$.

We will find it useful to consider a particular decomposition of $M(H^\infty + C)$, namely we will identify points in $M(H^\infty + C)$ that agree on all $QC$ functions to obtain the $QC$-level sets. 
 If $x \in M(L^\infty)$ we let 
$$
E_x = \{y \in M(L^\infty) : y(q) = x(q)~\mbox{for all}~ q \in QC\}
$$ 
denote the $QC$-level set corresponding to $x$. The Bishop decomposition, page 60 of  \cite{Gamelin}, says that a function $f \in L^\infty$ is in the algebra $QC$ if and only if $f|E_x$ is constant for all $QC$-level sets $E_x$.

\section{P. Beurling Functions}
\label{P. Beurling}

In this section, our goal is to provide simpler proofs of certain results on interpolation with thin sequences. We will use two constructions, both of which produce a sequence of functions $\{f_k\}$ corresponding to our sequence $\{z_k\}$ such that $f_k(z_k) = 1$, $f_k(z_j) = 0$ if $k \ne j$ and $\sum_k |f_k(z)| < M$ for all $z \in \mathbb{D}$. We call these functions, $f_k$, P. Beurling functions. Peter Jones \cite{J} has shown how to construct such functions corresponding to interpolating sequences in $H^\infty$. Since we will use a slight modification of the Jones construction, we state the results we will use below without proof. The details can be found in \cite{J}. We also provide a simpler proof that asymptotic interpolating sequences of type $1$ are thin sequences.

\begin{lm}\label{estimates}
Let $\{z_j\}$ be an interpolating sequence.  Let  $$
g_j(z):= \frac{B_j(z)}{B_j(z_j)}\left(\frac{1-\abs{z_j}^2}{1-\overline{z_j}z}\right)^2\exp\left(-\sum_{\abs{z_m}\geq\abs{z_j}}\left(\frac{1+\overline{z_m}z}{1-\overline{z_m}z}-\frac{1+\overline{z_m}z_j}{1-\overline{z_m}z_j}\right)(1-\abs{z_m}^2)\right).
$$ 
Then there exists a constant $C(\delta)$ depending only on the separation constant $\delta := \min\{\delta_j: j \ge 1\}$ such that $$
\abs{g_j(z)}\leq C(\delta) \left(\frac{1-\abs{z_j}^2}{\abs{1-\overline{z_j}z}}\right)^2\exp\left(-\sum_{\abs{z_m}\geq\abs{z_j}}\left(\frac{1-\abs{z_m}^2}{\abs{1-\overline{z_m}z}}\right)^2\right).
$$

\end{lm}


As in \cite{J}, this can be used to obtain a concrete description of a function that does the interpolation, as indicated below. However this function is not of minimal norm. To get the functions of minimal norm, we must do something different. For that, we turn to the commutant lifting theorem and P. Beurling functions. First, let us state the interpolation
result for $H^\infty$ by means of Jones' functions from Lemma \ref{estimates}.

\begin{thm}
\label{JonesInterp_asymp}
Suppose that the sequence $Z$ is a thin sequence, i.e., $\lim_{j\to\infty}\delta_j=1$, where $\delta_j:=\left\vert B_j(z_j)\right\vert$.  
Define 
$$
g_j(z):= \frac{B_j(z)}{B_j(z_j)}\left(\frac{1-\abs{z_j}^2}{1-\overline{z_j}z}\right)^2\exp\left(-\sum_{\abs{z_m}\geq\abs{z_j}}\left(\frac{1+\overline{z_m}z}{1-\overline{z_m}z}-\frac{1+\overline{z_m}z_j}{1-\overline{z_m}z_j}\right)(1-\abs{z_m}^2)\right),
$$
if $\delta_j \ne 0$ and take $g_j = 0$ otherwise.

Then there exists $N$ such that for any $a\in\ell^\infty$ 
$$
g(z)=\sum_{j=1}^\infty a_j g_j(z)\in H^\infty
~\mbox{with}~g(z_j) = a_j~\mbox{for}~j \ge N$$ and
$$
\sup_{z\in\D}\abs{g(z)}\leq C(\delta)\norm{a}_{\ell^\infty}.
$$
where $C(\delta)$ is a constant depending on $\delta:= \min\left\{\delta_j: j \ge N\right\}$.

\end{thm}

These functions are extremely useful in interpolation theory, but because of the constant appearing above, they are not sharp in the sense of the following 
beautiful theorem of Per Beurling (see  e.g. page 285 in Garnett):

\begin{thm}\label{thm:pbeurling} Let $\{z_j\}$ be an interpolating sequence in the upper half plane.
Let $$M = \sup_{|a_j| \le 1} \inf\left\{\|f\|_\infty: f \in H^\infty, f(z_j) = a_j, j \in\N \right\}.$$ Then there are functions $f_j \in H^\infty$ such that $f_j(z_j) = 1, f_j(z_k) = 0~\mbox{for}~k \ne j$ and for each $z$,
$$\sum_j |f_j(z)| \le M.$$\end{thm}
For versions of Theorem \ref{JonesInterp_asymp} with optimal constant, see \cite{NOCS} or \cite{PTW}.

Before stating the main result of this section, we recall the Commutant Lifting Theorem.  Given a Hilbert space $\Hc$ we let 
$$
H_{\Hc}(\mathbb{D}) = \left\{f : \mathbb{D} \to \Hc: f(z) = \sum_{n=0}^\infty a_n z^n, \|f\|_2^2:= \sum_{n=0}^\infty\left\| a_n\right\|_\Hc^2\right\}.
$$ 
Let $S_{\Hc}$ denote multiplication by $z$ on $H_{\Hc}^2$. We write $X \leftrightarrow Y$ to indicate that two operators commute.

\begin{thm}[Commutant Lifting Theorem for $H_{\Hc}^2$]
\label{thm:commlift}
Let $M \subseteq H_{\Hc}^2$ be an invariant subspace for $S_{\Hc}$ and suppose $X \in B(M)$ and $X^\star \leftrightarrow S_{\Hc}^\star|_{M}$. Then there exists $Y \in B(H_{\Hc}^2)$ such that 
\begin{enumerate}
\item $Y^\star|_{M} = X^\star$ $(X^\star$ lifts to $Y^\star)$;
\item $Y$ is in the commutant of $S_\mathcal{H}$;
\item $\|Y\| = \|X\|$.
\end{enumerate}
\end{thm}

We use  Theorem \ref{thm:pbeurling} and Theorem \ref{thm:commlift} to obtain
\begin{thm}\label{thm:replacement} 
Let $\{z_n\}$ be a thin sequence. Then for every $\varepsilon > 0$ there exists $N$  such that for $n \ge N$ there exist $f_n \in H^\infty$ such that for $j, k \ge N$ we have $$f_n(z_n) = 1, f_n(z_k) = 0, j \ne k,$$ and for all $z \in \mathbb{D}$ we have $$\sum_{n \ge N} |f_n(z)| < (1 + \varepsilon).$$ In particular, for every sequence $a \in \ell^\infty$ with $\|a\|_{\ell^\infty} \le 1$ the function $g_a$ defined by $g_a(z):= \sum_{n \ge N} a_n f_n(z) \in H^\infty$, $\|g_a\|_\infty \le  (1 + \varepsilon)\|a\|_{N, \ell^\infty}$ and $g_a(z_j) = a_j$ for $j \ge N$.
\end{thm}


\proof The sequence $\{z_n\}$ is eventually interpolating, so
by Theorem \ref{Volberg},  
there exist a separable Hilbert space $\mathcal{K}$, an orthonormal basis $\{e_n\}_{n \ge k}$ for $\mathcal{K}$ and $U, K: \mathcal{K} \to K_B$, $U$ unitary, $K$ compact, $U + K$ invertible, such that 
$(U + K)(e_n) = h_{z_n}$   for $n \ge k$.       Let $\varepsilon > 0$. Hence there exists $N \ge k$ such that $K$ has norm less than $\varepsilon $ on 
$\mathcal{K}_N = \overline{\textnormal{span}\{e_n: n\ge N\}}$. Let $K^N_B= \overline{\textnormal{span}\{h_{z_n}: n\ge N\}} \subset K_B$.     Then
$$
   U+K : \mathcal{K}_N \rightarrow K^N_B
$$
is  invertible with an inverse of norm less than $\frac{1}{1-\varepsilon}$. Now for $a \in l^\infty$, define the operator
$$
   T^*_a : K^N_B  \rightarrow K^N_B, \quad h_{z_n} \mapsto \bar a_n h_{z_n}    \quad (n \ge N).
$$
Writing
$$
     T^*_a =   (U+K)       D^*_a (U+K)^{-1}
$$
with 
$$
    D^*_a: \mathcal{K}_N \rightarrow \mathcal{K}_N, \quad e_n \mapsto \bar a_n e_n    \quad (n \ge N),
$$
we see that
$$
   \|   T^*_a \|_{K^N_B \to K^N_B} <   \frac{1+ \varepsilon}{1-\varepsilon} \|a\|_{\ell^\infty}.
$$
The normalized reproducing kernels $h_{z_n}$ are eigenvectors of the adjoint of the shift operator on $H^2$, denoted $S^*$. Therefore, $K^N_B$ is invariant under $S^*$,
and the restriction of $S^*$ to $K^N_B$ commutes with $T^*_a$. By the Commutant Lifting Theorem, Theorem    \ref{thm:commlift}, there exists an operator
$T^*:H^2 \rightarrow H^2$ extending $T^*_a$ that commutes with $S^*$ and satisfies $\|T^*\|_{2 \to 2} = \|T^*_a\|_{K^N_B   \to K^N_B}$. Since
$T^*$ commutes with $S^*$, it is the adjoint of a Toeplitz operator $T_\phi$ with symbol $\phi \in H^\infty$ and 
$$
     \| \phi\|_\infty = \|T_\phi\|_{2 \to 2}= \|T^*\|_{2 \to 2} = \|T^*_a\|_{K^N_B   \to K^N_B} < \frac{1+ \varepsilon}{1-\varepsilon} \|a\|_{\ell^\infty}.      
$$
Since $T^*$ extends $T^*_a$, and the $ h_{z_n} $ are eigenvectors of $T_\phi^*$ with eigenvalue $\bar \phi(z_n)$,    we obtain
$$
   \bar \phi (z_n)  h_{z_n} =   T_\phi^*  h_{z_n}  =    T_a^*  h_{z_n}  = \bar a_n  h_{z_n}    \quad \text{ for all } n \ge N
$$
and consequently 
$$
      \phi (z_n) = a_n  \quad \text{ for all } n \ge N,   \quad  \| \phi\|_\infty < \frac{1+ \varepsilon}{1-\varepsilon} \|a\|_{\ell^\infty}.
$$
Since this applies to any sequence in $l^\infty$, we can apply Per Beurling's Theorem \ref{thm:pbeurling} for the sequence $\{z_n\}_{n \ge N} $ with constant
$M =  \frac{1+ \varepsilon}{1-\varepsilon}$      to obtain functions
$f_n \in H^\infty$ for $n \ge N$ such that for $j, k \ge N$ we have $f_n(z_n) = 1, f_n(z_k) = 0, j \ne k$, and
$$
\sum_{n \ge N}  |f_j(z)| \le     \frac{1+ \varepsilon}{1-\varepsilon} .
$$
In particular, for any $\{a_n\} \in l^\infty$,
 $$\sum_{n \ge N} |a_n f_n(z)| < (1 + 3 \varepsilon)\|a\|_{N, \ell^\infty}.$$
\qed

The following definition in $H^p$ will play a significant role in Section~\ref{asip}. For now we consider just the case $p = \infty$.

\begin{defin}
We say that a sequence $\{z_n\}$ is an eventual $1$-interpolating sequence for $H^\infty$ $(EIS_\infty)$ if the following holds: For every $\varepsilon > 0$ there exists $N$ such that for each $a \in \ell^\infty$ there exists $f_N \in H^\infty$ with
$$f_N(z_n) = a_n ~\mbox{for}~ n \ge N ~\mbox{and}~ \|f_N\|_\infty \le (1 + \varepsilon) \|a\|_{N, \ell^\infty}.$$
\end{defin} 
Putting all of this together, we obtain the following corollary.

\begin{cor}    \label{cor:eip}    A sequence $\{z_n\}$ is thin if and only if it is an $EIS_\infty$ sequence.\end{cor}
\begin{proof} If $\{z_n\}$ is thin, then it is eventually interpolating, and an $EIS_\infty$ sequence by Theorem \ref{thm:replacement}. Conversely, if $\{z_n\}$ is an $EIS_\infty$
sequence, for each $\varepsilon$ there exists an $N \in \N$ such that for each $j \ge N$, there exists $f_j \in H^\infty$ with
$$
         f_j(z_n) = 0 \text{ for all } n \ge N, n \neq j, \quad f_j(z_j)=1  \text{ and } \|f_j\|_\infty < 1 + \varepsilon.
$$
Letting $B_{j,N}$ denote the Blaschke product for the sequence $\{z_n\}_{n \ge N, n \neq j}$, we obtain that $f = B_{j,N} g_j$, where $g_j \in H^\infty$
with $\|g_j\|_\infty < 1 + \varepsilon$. It follows that 
$$
                   \Pi_{n \ge N, n \neq j} \left|\frac{z_n - z_j}{1 - \bar z_j z_n} \right| =          |B_{j,N}(z_j)| \ge \frac{1}{1 + \varepsilon} \text{ for all } j \ge N
$$
and consequently $\{z_n\}$ is thin.
\end{proof}

Hence, as we will show in a moment, Theorem~\ref{newDN} below will provide a new proof of Theorem~\ref{thm:DN}.

\begin{thm}[Dyakonov, Nicolau \cite{DN}]\label{thm:DN}
Let $\{z_j\}$ be an interpolating sequence. Then the following are equivalent:
\begin{enumerate}
\item $\{z_j\}$ is thin;
\item There is a sequence $m_j \in (0, 1)$ with $\lim_{j \to \infty} m_j = 1$ such that every interpolation problem $F(z_j) = a_j$ with $|a_j| \le m_j$ for all $j$ has a solution in $H^\infty$ with $\|F\|_\infty \le 1$.
\end{enumerate}
\end{thm}

\begin{thm}\label{newDN} Let $\{z_n\}$ be an interpolating sequence. Then $\{z_n\}$ is an $EIS_\infty$ sequence if and only if there exists a sequence $m_j \in (0, 1)$ with $m_j \to 1$ such that any sequence $\{a_j\}$ satisfying $|a_j| \le m_j$ for all $j$ can be interpolated by an $H^\infty$ function of norm $1$.
\end{thm}

\begin{proof}
First suppose the sequence has associated constants $m_j$. Let $\varepsilon > 0$. There exists $N$ such that $m_j > \frac{1}{1 + \varepsilon}$ for $j \ge N$. Let $a \in \ell^\infty$ be such that $\|a\|_{\ell^\infty} \le 1$. 
For $j \ge N$ we have $\frac{a_j}{(1 + \varepsilon)\|a\|_{N, \ell^\infty}} \le m_j$ and therefore there exists an $H^\infty$ function $F$ of norm $1$ with $F(z_j) =  \frac{a_j}{(1 + \varepsilon)\|a\|_{N, \ell^\infty}}$. The function $F_1 = (1 + \varepsilon)\|a\|_{N, \ell^\infty}\,F$ does the interpolation $F_1(z_j) = a_j$ and $\|F_1\| \le (1 + \varepsilon)\|a\|_{N, \ell^\infty}$, as desired.

\bigskip
For the converse we first introduce our notation, which will allow us to choose the $m_j$. \\

Choose $\varepsilon_j \to 0$. Let $b_{n-}$ denote the Blaschke product with zeros $\{z_j : j < n\}$. Then $b_{n-}(z_k) = 0$ if $k < n$ and $|b_{n-}(z_k)| \ge \delta_k$ if $k \ge n$. If we let $\delta_{n}^\prime = \inf_{k \ge n} \delta_k$, then $\delta_{n}^\prime \to 1$ as $n \to \infty$.

 Let $a_k = 0$ if $k < n$ and $a_k = \frac{1}{b_{n-}(z_k)}$ otherwise. 
 By assumption, given $\varepsilon > 0$ there exists $N$ and $f \in H^\infty$ such that $f(z_k) = a_k$ for $k \ge N$ and $\|f\|_{\infty} \le (1 + \varepsilon) \|a\|_{N, \infty}$. Therefore, multiplying yields $b_{n-}f$ satisfying $\|b_{n-} f\|_\infty \le \frac{1 + \varepsilon}{\delta_n^\prime}$, $b_{n-}f(z_j) = 1$ if $j \ge \max\{n, N\}$ and $b_{n-}f(z_j) = 0$ if $j < n$. So, if we choose $n \ge N$, we have a function $b_{n-} f$ that vanishes on $z_k$ for $k < n$, is equal to $1$ for $k \ge n$, and has norm at most $\frac{1 + \varepsilon}{\delta_n^\prime}$. Now let $t_n = \frac{\delta_n^\prime}{1 + \varepsilon_n}$ and $k_n = \frac{\delta_n^\prime b_{n-} f}{1 + \varepsilon_n} = t_n b_{n-}f$. Note that $\|k_n\|_\infty \le 1$.

Let $\rho$ denote the pseudohyperbolic metric on $\mathbb{D}$, $\rho(z,w) = |\frac{z-w}{1- \bar w z}|$ for $z,w \in \mathbb{D}$.
Choose $\delta_{t, n} \in (0, 1)$ so that 
$$\rho(-1 + \delta_{t, n}, 1 - \delta_{t, n}) = t_n = \rho(0, t_n),$$ 
and note that $\delta_{t, n} \to 0$. Now there is a M\"obius transformation $\varphi$ with $\varphi(0) = -1 + \delta_{t, n}$ and $\varphi(t_n) = 1 - \delta_{t, n}$, so
$$\varphi(k_n(z_j)) = -1 + \delta_{t, n} ~\mbox{for}~ j < n$$ and
$$\varphi(k_n(z_j)) = 1 - \delta_{t, n}~\mbox{for}~j \ge n.$$ Let $$h_n = \frac{\varphi \circ k_n}{1 - \delta_{t, n}}, F_n = \left(\frac{1 + h_n}{2}\right)^2~\mbox{and}~G_n = \left(\frac{1 - h_n}{2}\right)^2.$$ Then 
$$F_n(z_j) = 1~\mbox{for}~j \ge n~\mbox{and}~ F_n(z_j) = 0~\mbox{for}~ j < n,$$  
$$G_n(z_j) = 0~\mbox{for}~j \ge n~\mbox{and}~G_n(z_j) = 1~\mbox{for}~j < n$$ and $$|F_n| + |G_n| < \frac{1}{2} + \frac{1}{2(1 - \delta_{t,n})^2} := \gamma_n.$$

\bigskip
Now for each $\varepsilon_j$ we obtain $N_j$ and $b_{N_j-}$ giving us $\delta_{N_j}$ which, in turn, gives us $\delta_{t, N_j}$. We may choose $\varepsilon_j$ tending to $0$ quickly and $N_j$ as large as we like to ensure that $\delta_{t, N_j}$ is close to $1$. So,  we will  choose a subsequence of $\varepsilon_j$, which we denote by $\varepsilon_j$ again and corresponding $N_j$ so that the $\gamma_j$ as defined above satisfy $\prod_{j = 1}^\infty \gamma_j$ converges to a positive number. Now we will construct our sequence $\{m_j\}$.\\

{\bf Stage 1.} For $n < N_1$ we use the Jones construction to interpolate all sequences of norm smaller than $\frac{1}{C(\delta)}$ with a function of norm at most $1$. Call this function $f_1$.\\

{\bf Stage 2.} Now we know there is a function of norm at most $1 + \varepsilon_1$ that does interpolation on sequences of norm at most $1$ from $N_1$ on. Divide this function by $1 + \varepsilon_1$, call it $f_2$. Choose $F_1 := F_{N_1}$ and $G_1:= G_{N_1}$ as above and note that, with the corresponding $\gamma_1$ defined as above, we have $$|F_1| + |G_1| < \gamma_1.$$

So we see that
\begin{enumerate}
\item $(f_1 G_1 + f_2 F_1)(z_j) = f_2(z_j)$ for $j \ge N_1$;
\item $(f_1 G_1 + f_2 F_1)(z_j) = f_1(z_j)$ for $j < N_1$;
\item  $|f_1 G_1 + f_2 F_1| \le |G_1| + |F_1| <  \gamma_1$.
\end{enumerate}

\bigskip
We take $g_2 = \frac{f_1 G_1 + f_2 F_1}{\gamma_1}$. Then $\|g_2\|_\infty \le 1$ and $g_2$ can interpolate \\

\begin{enumerate}
\item[a)] $|a_j| \le \frac{1}{\gamma_1 C(\delta)}$ if $j < N_1$,\\ 
\item[b)] $|a_j| \le \frac{1}{\gamma_1 (1 + \varepsilon_1)}$ if $j \ge N_1$,\end{enumerate}
if we choose appropriate $f_1$ and $f_2$.
\\

{\bf Stage 3.} Now we choose $f_3$ of norm $1$ doing interpolation on sequences of norm smaller than $\frac{1}{1 + \varepsilon_2}$ from $N_2$ on. Then we choose $G_2:= G_{N_2}$ and $F_2 := F_{N_2}$ as above with $$|G_2| + |F_2| < \gamma_2.$$ Consider 
$$g_3:= \frac{g_2 G_2 + f_3 F_2}{\gamma_2} = \frac{f_1}{\gamma_1 \gamma_2} G_1G_2 + \frac{f_2}{\gamma_1 \gamma_2} F_1 G_2 + \frac{f_3}{\gamma_2} F_2 .$$ Then 
$$\|g_3\|_\infty \le \frac{|G_2| + |F_2|}{\gamma_2} \le 1,$$ and

$$g_3(z_j) = \frac{f_1(z_j)}{\gamma_1\gamma_2} ~\mbox{for}~ j < N_1,$$
$$g_3(z_j) = \frac{f_2(z_j)}{\gamma_1\gamma_2}~\mbox{for}~N_1 \le j < N_2,$$
$$g_3(z_j) = \frac{f_3(z_j)}{\gamma_2}~\mbox{for}~j \ge N_2.$$

So $g_3$ can interpolate sequences satisfying 

\begin{enumerate}
\item[a$^\prime$)] $|a_j| \le \frac{1}{\gamma_1 \gamma_2 C(\delta)}$ for $j < N_1$; \\
\item[b$^\prime$)] $|a_j| \le \frac{1}{\gamma_1 \gamma_2 (1 + \varepsilon_1)}$ for $N_1 \le j < N_2$;\\
\item[c$^\prime$)] $|a_j| \le \frac{1}{\gamma_2(1 + \varepsilon_2)}$ for $j \ge N_2$.
\end{enumerate} 

%
\bigskip
We repeat this process arriving at Stage $n$.

\bigskip

{\bf Stage $n$.} Consider $\varepsilon_j$, $N_j$, the corresponding $\frac{f_j}{1 + \varepsilon_{j - 1}}$ obtained from the $EIS_\infty$ assumption that interpolate $z_j$ to $a_j$ for $N_{k - 1} \le j < N_k$ if $|a_j| \le \frac{1}{1 + \varepsilon_{j - 1}}$. Construct, for each j, the functions $F_j := F_{N_j}$ and $G_j := G_{N_j}$  so that $$|F_j| + |G_j| <  \gamma_j.$$ Finally, define
\begin{eqnarray*}
g_n & = &  \left(\frac{f_1}{\prod_{j = 1}^{n - 1} \gamma_j}\right) G_1 \cdots G_{n - 1} + \left(\frac{f_2}{\prod_{j = 1}^{n - 1}\gamma_j}\right) F_1 G_2 \cdots G_{n - 1} + \\
& + & \left(\frac{f_3}{\prod_{j = 2}^{n - 1} \gamma_j}\right) F_2 G_3 \cdots G_{n - 1} + \cdots + \left(\frac{f_{n - 1}}{\gamma_{n - 2}\gamma_{n -1}}\right) F_{n - 2} G_{n - 1} + \left(\frac{f_{n}}{\gamma_{n - 1}}\right) F_{n - 1}\end{eqnarray*} so that (taking $N_0 = 0$)
\begin{enumerate}
\item $g_n(z_j) = \frac{f_1(z_j)}{\prod_{j = 1}^{n - 1} \gamma_j}$ for $j < N_1$;
\item $g_n(z_j) = \frac{f_k(z_j)}{\prod_{j = k-1}^{n - 1} \gamma_j}~\mbox{for}~N_{k-1} \le j < N_k$ where  $1 < k \le n - 1$;
\item $g_n(z_j) =  f_{n}(z_j)$ for $j \ge N_{n - 1}$;
\item $\|g_n\|_\infty \le 1$.
\end{enumerate}
Since for each $k>1$, the function $\frac{f_k}{\prod_{j = k - 1}^{n - 1}\gamma_j}$ interpolates the set 
$$
\left\{\{a_j\}: |a_j| \le \frac{1}{\left(\prod_{l = k - 1}^\infty \gamma_l\right) (1 + \varepsilon_{k - 1})}, N_{k - 1} \le j < N_k\right\},
$$ 
$\frac{f_1}{\prod_{j = 1}^{n - 1} \gamma_k}$ interpolates sequences $|a_j| \le \frac{1}{C(\delta)\prod_{j = 1}^{n - 1} \gamma_j}$ for $j < N_1$, and  $$\frac{1}{\left(\prod_{l = k - 1}^\infty \gamma_l\right) (1 + \varepsilon_{k - 1})} \to 1$$ as $k \to \infty$, so we see that a normal families argument now implies the result.
\end{proof}

\begin{rem} It is also possible to use Theorem~\ref{thm:DN} to obtain a proof of the existence of the P. Beurling functions that we use above. \end{rem}

\section{New Proofs of Old Results on Asymptotic Interpolation}
\label{asiinfty}

We now return to the asymptotic interpolation sequences that we discussed in the introduction to this paper. Asymptotic interpolation sequences were studied in \cite{HKZ} and, later, in \cite{GM}. The Jones construction provides a constructive proof that thin sequences are asymptotically interpolating, but it does not show that the sequence is an $AIS$ of type $1$. A simpler proof was provided by Dyakonov and Nicolau, \cite{DN}.  We now turn to a simpler proof of the converse; i.e. that an asymptotic interpolating sequence of type $1$ is a thin sequence.  The argument in \cite{GM} (in both directions) relied heavily on deep maximal ideal space techniques and T. Wolff's work appearing in \cite{W} and \cite{WT}. For our new proof, we use the fact that a thin sequence is eventually interpolating, a fact that follows easily from \cite{CG} (see also Theorem 4.1 in \cite{Garnett}), stated below, and the fact that the sequence must be eventually discrete, found in Lemma 4.1 in \cite{GM}. For general uniform algebras, this is proved in Theorem 1.6 in \cite{GM}.

\begin{thm}[\cite{CG}]
\label{CG} Let $Z = \{z_j\}$ be a sequence in the upper half plane. Then $Z$ is an interpolating sequence if disjoint subsets of $\{z_j\}$ have disjoint closures in $M(H^\infty)$. \end{thm}

Since we will assume that $\{z_j\}$ is asymptotically interpolating, the sequence must be eventually distinct, \cite{GM}. So supposing that no points of the sequence are repeated, choose disjoint subsets $X$ and $Y$ of $Z$. Then there exists $f \in H^\infty$ with $f(a_j) \to 0$ for $a_j \in X$ and $f(b_j) \to 1$ for $b_j \in Y$. It follows that $X$ and $Y$ must have disjoint closures.

As in Sundberg and Wolff \cite{SW}, we view the theorem below as turning approximate interpolation into exact interpolation. However, in our proof that every asymptotic interpolating sequence is thin, we will use the fact that if $f \in H^\infty$ has the property that $f(z_n) \to 0$ where $\{z_n\}$ is interpolating, then $\overline{B} f \in H^\infty + C$, where $B$ is the corresponding interpolating Blaschke product. This last result can be proved using the fact that the zero sequence of $B$ is an interpolating sequence and therefore the map $T: H^\infty \to \ell^\infty$ defined by $T(f) = \{f(z_n)\}$ is surjective (see the work of the first author in \cite{AG} or \cite{GIS}), but it's interesting to note that the functions in Jones's proof show how to change the approximate condition $f(z_n) \to 0$ to an exact condition $f(z_n) = 0$, and consequently can be also used to obtain this result:

\begin{thm}[\cite{AG}, \cite{GIS}] \label{ag} Let $\{z_n\}$ be an interpolating sequence for $H^\infty,$  $f \in H^\infty$ and $\{a_n\}$ a sequence with $|f(z_n) - a_n| \to 0$. If $B$ denotes the interpolating Blaschke product corresponding to $\{z_n\}$, then there exists a function $h \in H^\infty$ with $\overline{ B} h \in H^\infty + C$
such that $(f - h)(z_n) = a_n$. Consequently, if $f(z_n) \to 0$, then $\overline{B} f \in H^\infty + C$. \end{thm}

\begin{proof}  We write $$g_j(z) = \left(\frac{1-\abs{z_j}^2}{1-\overline{z_j}z}\right)^2\ \exp\left(-\sum_{\abs{z_m}\geq\abs{z_j}}\left(\frac{1+\overline{z_m}z}{1-\overline{z_m}z}-\frac{1+\overline{z_m}z_j}{1-\overline{z_m}z_j}\right)(1-\abs{z_m}^2)\right).$$ Precisely the same argument as that used to establish the estimates in Theorem~\ref{JonesInterp_asymp} show that $\sum_{j = 1}^\infty |g_j(z)| \le C(\delta)$ for all $z \in \mathbb{D}$.
Let 
$$h_j(z) = \sum_{n = N_j + 1}^{N_{j + 1}} (f(z_n) - a_n) \frac{B_n(z)}{B_n(z_n)} g_n(z)$$ 
where $|f(z_k) - a_k| < \frac{1}{2^j}$ for $N_j < k \le N_{j + 1}$. 
Note that $h_j(z_k) = 0$ if $k \le N_j$ or $k > N_{j + 1}$ and $h_j(z_k) = f(z_k) - a_k$ for $k$ satisfying $N_j < k \le N_{j + 1}$.

Now
$$|h_j(z)| \le\frac{1}{2^j \delta}\sum_{n = N_j + 1}^{N_{j + 1}} |g_n(z)|.$$ By Lemma~\ref{estimates}, we see that $\|h_j\|_\infty < \frac{1}{2^j\delta} C(\delta)$, where $\delta:=\delta(B)$ is the interpolating constant of $B$ and $C(\delta)$ is a constant depending only on $\delta$.  Let $h = \sum_{j = 1}^\infty h_j$. Then $h \in H^\infty$ and $(f - h)(z_j) = a_j$ for all $j$.\\

To show that $\overline{B} h \in H^\infty + C$, note that on the unit circle we have
$$h_j(z) = \sum_{n = N_j + 1}^{N_{j + 1}} (f(z_n) - a_n) \frac{B_n(z)}{B_n(z_n)} g_n(z) = \sum_{n = N_j + 1}^{N_{j + 1}} (f(z_n) - a_n) \frac{B(z)}{B_n(z_n)} \left(\overline{\left({\frac{z - z_n}{1 - \overline{z_n} z}}\right)} g_n(z)\right).$$ Since $\overline{\frac{z - z_n}{1 - \overline{z_n} z}} g_n(z) \in H^\infty + C$ we have $h_j  = B k_j$ for some $k_j \in H^\infty + C$. 
Further $\|k_j\|_\infty = \|h_j\|_\infty < \frac{1}{2^j \delta}C(\delta)$ implies that $\overline{B} h = \sum_{j = 1}^\infty k_j \in H^\infty + C$. 

Finally, applying this to the case $f(z_n) \to 0$ and $ a_n =0$ for all $n$, we obtain that $f-h= Bg $ for some $g \in H^\infty$ and consequently $\overline{B} f = g + \overline{B} h \in H^\infty + C$.

\end{proof}

This allows us to show how the Jones construction can be used to prove a result in \cite{W}. Our proof will also use the Chang-Marshall Theorem (\cite{C}, \cite{M}) and the fact that an $H^\infty$ function is in $QA$ if and only if it is constant on every $QC$-level set, however our analysis will give some indication of where the zero sequence $\{z_j\}$ in the hypothesis lies.  At the time Wolff wrote his thesis, there was a simpler uniform algebra proof available, but Wolff's lemma (Lemma 1.2 of \cite{W}), Jones's construction and Theorem~\ref{ag} have simplified it further still.

\begin{cor}[Wolff, \cite{W}]\hfill
\label{multiply}
\begin{enumerate}
\item[(a)] Let $\{z_j\}$ be an interpolating sequence with corresponding Blaschke product $b$. If $q \in QA$ satisfies $q(z_j) \to 0$, then $qb \in QA$. In fact, $q \overline{b}^n \in QC$ and $q b^n \in QA$ for every $n \in \N$.\\
\item[b)] Let $f \in L^\infty$. Then there exists a non-zero $q \in QA$ such that $qf \in QC$.
\end{enumerate}
\end{cor}

We show that (b) is true in stages. First, we prove it for the conjugate of a Blaschke product, then for functions in $H^\infty + C$, and finally for functions in $L^\infty$. We will need two auxiliary results. First, as a consequence of the Chang-Marshall construction (\cite{C}, \cite{M}), we know that given any Blaschke product $B$, there is an interpolating Blaschke product $b$ such that the (closures of) the algebras generated by $H^\infty$ and the conjugates of the Blaschke products coincide; i.e. 
$$H^\infty[\overline{B}] = H^\infty[\overline{b}].$$ Note that for a $QC$-level set $E_x$ we know that $H^\infty + C|E_x = H^\infty|E_x$ is closed, and by Shilov's theorem, found on page 60 in \cite{Gamelin}, we know that if $f \in L^\infty$ and $f|E_x \in H^\infty|E_x$ for every $QC$-level set $E_x$, then $f \in H^\infty + C$. We turn to the proof of Corollary~\ref{multiply}.

\begin{proof} We begin with the proof of (a). From the above, we know that $q \overline{b} \in H^\infty + C$. Since $H^\infty + C$ is an algebra, $\overline{q} b \in H^\infty + C$. So $q \overline{b} \in QC$. Since $q$ is constant on each $QC$-level set, if $b$ is not constant we must have $q = 0$. If $b$ is constant, then $qb$ is constant as well. Therefore, $qb$ is constant on every level set and $qb \in QA$. Since $b^n$ is constant whenever $b$ is, we also see that $q \overline{b}^n \in QC$ and $q b^n \in QA$ for every $n$.

For (b), consider a Blaschke product $B$.  Now, by the Chang-Marshall theorem, $H^\infty[\overline{B}] = H^\infty[\overline{b}]$ for some interpolating Blaschke product $b$. Let $\{z_n\}$ denote the zero sequence of $b$ and use Wolff's Lemma 1.2 (\cite{W}) and part a):  If $\{z_n\}$ is a Blaschke sequence, there is an outer function $q \in QA$ such that $q(z_n) \to 0$. From part (a), since $q \in QA$ tends to zero on the zero sequence of $b$, we know $q b^n \in QA$ and $q|E_x = 0$ on any $QC$-level set on which $b$ is not constant. Suppose $b|E_x$ is constant. There exist $h_j \in H^\infty$ such that $\overline{B}|E_x = \sum_j h_j \overline{b}^j |E_x \in H^\infty|E_x$ for each $x$, since $H^\infty|E_x$ is closed and we assume $b$ is constant on $E_x$. Therefore $\overline{q} \overline{B}|E_x \in H^\infty|E_x$ if $b|E_x$ is constant. On the other hand, if $b|E_x$ is not constant, $\overline{q}|E_x = 0$ and we also have $\overline{q}\overline{B}|E_x \in H^\infty|E_x$. Thus, $\overline{qB}\in H^\infty +C$ and, of course, $qB \in H^\infty$ implies $qB \in QA$. 

To obtain the result for $f \in L^\infty$, we apply a result of Axler \cite{A1} that says that given $f \in L^\infty$, there exists a Blaschke product $B$ such that $B f \in H^\infty + C$.  Suppose that $g \in H^\infty + C$. By Axler's result, we can find a Blaschke product $B$ with $B \overline{g} \in H^\infty + C$, and from our work above we can find $q \in QA$ with $qB \in QA$.  Therefore $q B \overline{g} = q(B \overline{g}) \in H^\infty + C$ and $\overline{q B \overline{g}} = \overline{qB} g \in H^\infty + C$. So $\overline{qB} g  \in QC$ and we can multiply the conjugate of any $H^\infty + C$ function into $QC$. 

If we have $f \in L^\infty$, we can multiply by a Blaschke product $B_0$ and we have $B_0 f \in H^\infty + C$. We can now multiply the conjugate of this function into $QC$, so we have $q_0 \overline{B_0 f}$ and $\overline{q_0}(B_0 f)$ in $QC$. If $B_0 f$ is not constant on a level set $q_0$ must be zero, and if $B_0 f$ is a non-zero constant $q_0$ must be as well. Therefore $q_0 B_0 f \in QC$.  If $q_0 B_0 \notin QA$, we may multiply it into $QA$ with another $QA$ function $q_1$. Therefore
$q_1 q_0 B_0$ multiplies $f$ into $QC$.
\end{proof}

Note that our $q$ is obtained by taking $B$, finding an interpolating Blaschke product corresponding to it and requiring that $q$ tends to zero on the zeros of the interpolating Blaschke product.

Recall that a sequence is an asymptotic interpolation  sequence of type $1$ for $H^\infty$,
 if whenever $a \in \ell^\infty$, there exists $g \in H^\infty$ such that $|g(z_j) - a_j| \to 0$ and $\|g\|_\infty \le \|a\|_{\ell^\infty}$.
\begin{thm}[\cite{GM}]
\label{thm:asiinfty} A sequence $\{z_n\}$ of distinct points is an asymptotic interpolating sequence of type $1$ for $H^\infty$ if and only if $\{z_n\}$ is a thin interpolating sequence. \end{thm}

\begin{proof} We have already mentioned that a simpler proof that thin implies asymptotic of type $1$ follows from Theorem~\ref{thm:DN} and can be found in \cite{DN}. We now turn to the converse.

By Theorem~\ref{CG} and the comments following it, we may assume that our sequence is eventually interpolating. Since the sequence is assumed to be distinct, it must be interpolating.  Suppose that it is not a thin sequence. Then there is a sequence $\{z_{n_k}\}$ and $\varepsilon > 0$ such that $$|B_{n_k}(z_{n_k})| \le 1 - \varepsilon;$$ that is, $$\prod_{j \ne n_k} \rho(z_{n_k}, z_j) \le 1 - \varepsilon.$$ Choose a thin subsequence of the $\{z_{n_k}\}$ and denote this sequence by $\{z_{n_{k_m}}\}$. Let $w_n = 0$ if $n \notin \{n_{k_l}\}$ and $w_n = 1$ otherwise. Then choose $f \in H^\infty$ with $\|f\|_\infty = 1$ and $$|f(z_m) - w_m| \to 0.$$ So there exists $M$ such that $$f(z_{n_{k_m}}) \approx 1, f(z_j) \approx 0~\mbox{if}~ j \notin \{n_{k_m}\}$$ for $j, n_{k_m} \ge M$.  In fact, $f(z_j) \to 0$ as $j \to \infty$ for $j \notin \{n_{k_m}\}$ and $f(z_{n_{k_m}}) \to 1$ as $m \to \infty$. 

Let $B_1$ denote the Blaschke product with zeros $\{z_j\}_{j \notin \{n_{k_m}\}}$. By Theorem~\ref{ag} there exists $g \in H^\infty + C$ such that $f = B_1 g$ and $\|g\|_\infty = 1$. Thus $B_1, g \in H^\infty + C$, and by \cite{Douglas}
$$\lim_{|z| \to 1}|(B_1 g)(z) - B_1(z) g(z)| = 0,$$
where we interpret the evaluation of functions in $H^\infty + C$ on $\D$ via the Poisson extension formula.

Since $f(z_{n_{k_m}}) \to 1$, we see that $$\lim_m |1 - B_1(z_{n_{k_m}})g(z_{n_{k_m}})| = \lim_m |f(z_{n_{k_m}}) - B_1(z_{n_{k_m}}) g(z_{n_{k_m}})| = 0.$$ Therefore, $|B_1(z_{n_{k_m}})| \to 1$ as $m \to \infty$. Returning to the original Blaschke product $B$ and splitting the product into the terms belonging and not belonging to the
subsequence $\{n_{k_m}\}$, we obtain for all $ l \in \N$:
$$|B_{n_{k_l}}(z_{n_{k_l}})| = \prod_{m \ne l} \rho(z_{n_{k_m}}, z_{n_{k_l}}) \prod_{j \notin \{n_{k_m}\}} \rho(z_{n_{k_l}}, z_j) =  \prod_{m \ne l} \rho(z_{n_{k_m}}, z_{n_{k_l}})  \cdot |B_1(z_{n_{k_l}})|.$$ 
Now the first factor tends to $1$ as $l \to \infty$, because we have chosen $\{z_{n_{k_m}}\}$ to be a thin sequence, and we have just seen that the second factor tends to $1$ as well. This contradicts our choice of $\{z_{n_{k_m}}\}$, so the sequence must be thin. \end{proof}

\section{Thin Sequences in \texorpdfstring{$H^p$}{Hardy Spaces}} 

\label{asip}

In this section, we study $H^p$ equivalences for a sequence to be thin. Recall that for $z_0 \in \mathbb{D}$,  $k_{z_0}(z) = \frac{1}{(1 - \overline{z_0}z)}$ is the reproducing kernel for $H^2$ and $h_{z_0}(z) = \frac{\sqrt{1 - |z_0|^2}}{1 - \overline{z_0} z}$ is the normalized reproducing kernel.

\begin{defin}
Let $1 \le p \le \infty$. A sequence $\{z_n\}$ is an eventual $1$-interpolating sequence  for $H^p$ $(EIS_p)$ if the following holds: For every $\varepsilon > 0$ there exists $N$ such that for each $\{a_n\} \in \ell^p$ there exists $f_{N, a} \in H^p$ with
$$f_{N, a}(z_n) (1 - |z_n|^2)^{1/p} = a_n ~\mbox{for}~ n \ge N ~\mbox{and}~ \|f_{N, a}\|_p \le (1 + \varepsilon) \|a_n\|_{N, \ell^p}.$$
\end{defin} 

\begin{defin} Let $1 \le p \le \infty$. A sequence $\{z_j\}$ is a strong $AIS_p$-sequence if for all $\varepsilon > 0$ there exists $N$ such that for all sequences $\{a_j\} \in \ell^p$ there exists a function $G_{N, a} \in H^p$ such that $\|G_{N, a}\|_p \le \|a\|_{N,\ell^p}$ and 
$$\|G_{N, a}(z_j) (1 - |z_j|^2)^{1/p} - a_j\|_{N, \ell^p} < \varepsilon \|a_j\|_{N, \ell^p}.$$ \end{defin}

\begin{thm}\label{EISiffASI} 
Let $\{z_n\}$ be a sequence of  points in $\mathbb{D}$. Let $1 \le p \le \infty$. Then $\{z_n\}$ is an $EIS_p$ sequence if and only if $\{z_n\}$ is a strong-$AIS_p$.
\end{thm}

\begin{proof} 
If a sequence is an $EIS_p$, then it is trivially $AIS_p$, for given $\varepsilon > 0$ we may take $G_{N, a} = \frac{f_{N, a}}{(1 + \varepsilon)}$. 

For the other direction, suppose $\{z_n\}$ is an $AIS_p$ sequence. Let $\varepsilon > 0$, $N := N(\varepsilon)$, and  $\{a_j\}:=\{a_{j}^{(0)}\}$ be any sequence. First choose $f_0$ so that 
$$\|(1 - |z_j|^2)^{1/p}f_0(z_j) - a_{j}^{(0)}\|_{N, \ell^p} < \varepsilon \|a\|_{N, \ell^p}$$ and $$\|f_0\|_p \le \|a\|_{N,\ell^p}.$$
Now let $a_{j}^{(1)} = a_{j}^{(0)} - (1 - |z_j|^2)^{1/p} f_0(z_j)$. Note that $\|a^{(1)}\|_{N, \ell^p} < \varepsilon \|a\|_{N, \ell^p}$. Since we have an $AIS_p$ sequence, we may choose $f_1$ such that
$$\|(1 - |z_j|^2)^{1/p} f_1(z_j) - a_{j}^{(1)}\|_{N, \ell^p} < \varepsilon \|a^{(1)}\|_{N, \ell^p} <  \varepsilon^2\|a\|_{N, \ell^p},$$
and $$\|f_1\|_p \le \|a^{(1)}\|_{N, \ell^p}<\varepsilon\|a\|_{N,\ell^p}.$$ In general, we let
$$a_{j}^{(k)} = -(1 - |z_j|^2)^{1/p}f_{k - 1}(z_j) + a_{j}^{(k-1)}$$ so that
$$\|a^{(k)}\|_{N, \ell^p} \le \varepsilon \|a^{(k - 1)}\|_{N, \ell^p} \le \varepsilon^2 \|a^{(k-2)}\|_{N, \ell^p} \le \cdots \le \varepsilon^k \|a\|_{N, \ell^p}$$ and 
$$\|f_k\|_p \le \|a^{(k)}\|_{N, \ell^p}<\varepsilon^k\|a\|_{N,\ell^p}.$$ 
Then consider $f(z) = \sum_{k = 0}^\infty f_k(z)$.
Since $f_k(z_j) = \left(a_{j}^{(k)} - a_{j}^{(k+1)}\right)(1 - |z_j|^2)^{-1/p}$ and $a_{j}^{(k)} \to 0$ as $k \to \infty$, we have $$f(z_j) = a_{j}^{(0)}(1 - |z_j|^2)^{-1/p} = a_j(1 - |z_j|^2)^{-1/p}.$$ Further $\|f\|_p \le \sum_{k = 0}^\infty \varepsilon^{k} \|a\|_{N, \ell^p} = \frac{1}{1 - \varepsilon} \|a\|_{N, \ell^p}$.
\end{proof}

We will now turn to characterization of thin sequences by means of Carleson measures.  For $z \in \mathbb{D}$, we let $I_z$ denote the interval in $\T$ with center $\frac{z}{|z|}$ and length $1 - |z|$. For an interval $I$ 
in $\T$, we let $$S_I = \left\{z \in \mathbb{D}: \frac{z}{|z|} \in I~\mbox{and}~ |z| \ge 1 - |I|\right\}.$$ For $A > 0$, the interval $AI$ denotes an interval with the same center as $I$ and length $A|I|$.
Given a positive measure $\mu$ on $\D$, let us denote the (possibly infinite) constant
$$
    \CC(\mu) =  \sup_{f \in H^2, f \neq 0} \frac{\|f\|^2_{L^2(\D, \mu)}}{\|f\|^2_2}
$$
as the Carleson embedding constant of $\mu$ on $H^2$ and
$$
    \RR(\mu) =  \sup_{z} \frac{\|k_z\|_{L^2(\D, \mu)}}{\|k_z\|_2}
$$
as the embedding constant of $\mu$ on the reproducing kernel of $H^2$.

The Carleson Embedding Theorem asserts that the constants are equivalent. In particular, there exists a constant $c$ such that
$$
   \RR(\mu) \le \CC(\mu) \le c \RR(\mu),
$$
with best known constant $c= 2e$ \cite{PTW}.

 We recall the following result  from \cite{SW},
for a generalized version, see \cite{CFT}. 
\begin{thm}[See either Sundberg, Wolff, Lemma 7.1 in \cite{SW}, or Chalendar, Fricain, Timotin, Proposition 4.2 in \cite{CFT}]
\label{Cmeasure} Suppose $Z = \{z_n\}$ is a sequence of distinct points. Then the following are equivalent:

\begin{enumerate} 
\item $Z$ is a thin interpolating sequence;
\item for any $A \ge 1$, $$\lim_{n \to \infty} \frac{1}{|I_{z_n}|} \sum_{k \ne n, z_k \in S(A I_n)} (1 - |z_k|) = 0.$$
\end{enumerate}
\end{thm}

Here we note the following result.

\begin{thm}
\label{thm:Carleson} Suppose $Z = \{z_n\}$ is a sequence. For $N > 0$, let  $$\mu_N = \sum_{k \ge N} (1 - |z_k|^2)\delta_{z_k}.$$        
Then the following are equivalent:
\begin{enumerate} 
\item $Z$ is a thin sequence;
\item $ \CC(\mu_N) \to 1$ as $N \to \infty$;
\item $ \RR(\mu_N) \to 1$ as $N \to \infty$.
\end{enumerate}
\end{thm}
\begin{proof}
Noting that for each $f \in H^2$
$$
        \| f\|^2_{L^2(\D, \mu_N)} = \sum_{k=N}^\infty (1 - |z_k|^2) |f(z_k)|^2 = \sum_{k=N}^\infty |\langle f, h_{z_k}\rangle|^2,
$$
the implication (1) $\Rightarrow$ (2) follows immediately from Theorem \ref{Volberg}, and the implication (2) $\Rightarrow$ (3) is of course trivial. For (3) $\Rightarrow$ (1),
note first that (3) implies that  there exists $M$ such that for $N \ge M$ the sequence $\{z_n\}_{n \ge N}$ is an interpolating sequence, which we see from the following: For any $k \neq n$, $n,k \ge N$,
\begin{multline*}
  1- \rho(z_n, z_k)^2  = 1- \left|  \frac{z_k - z_n}{1 - \bar z_k z_n}\right|^2 = \frac{(1- |z_n|^2)(1- |z_k|^2)}{ | 1 - \bar z_k z_n|^2} \\ = (1 - |z_k|^2) | h_{z_n}(z_k)|^2 
   <  \|h_{z_n}\|^2_{L^2(\D, \mu_N )} -1  \stackrel{N \to \infty}{\rightarrow} 0,
\end{multline*}
so  $\{z_k\}_{k \ge N}$ is separated for sufficiently large $N$, and together with $ \RR(\mu_N) < \infty$, this implies that the sequence is interpolating from $N$ onwards
(see \cite{nikolski},
page 158) and, in particular, Blaschke.
By the Weierstrass Inequality, we obtain for $n \ge N$ that
\begin{multline*}
   \prod_{k \ge N, k \neq n}  \left|  \frac{z_k - z_n}{1 - \bar z_k z_n}\right|^2 =   \prod_{k \ge N, k \neq n} \left(  1-  \frac{(1 - |z_k|^2)(1 - |z_n|^2)}{|1 - \bar z_k z_n|^2} \right)   \\
       \ge 1 - \sum_{k \ge N, k \neq n}  \frac{(1- |z_n|^2)(1- |z_k|^2)}{ | 1 - \bar z_k z_n|^2}      
       =    1 -  \left(\|h_{z_n}\|^2_{L^2(\D, \mu_N )} -1   \right )    \stackrel{N \to \infty}{\rightarrow} 1,
\end{multline*}
which implies that $|B_n(z_n)| \rightarrow 1    $. Hence $\{z_n\}$ is thin.

\end{proof}

Putting the results above together, we arrive at our main result in this section.

\begin{thm}\label{main} Let $\{z_n\}$ be a Blaschke sequence of distinct points in $\mathbb{D}$. The following are equivalent:
\begin{enumerate} 
\item $\{z_n\}$ is an $EIS_p$ sequence for some $p$ with $1 \le p \le \infty$;
\item $\{z_n\}$ is thin;
\item $\{h_{z_n}\}$ is a complete AOB in $K_B$;
\item $\{z_n\}$ is a strong-$AIS_p$ sequence for some $p$ with $1 \le p \le \infty$;
\item The measure $$\mu_N = \sum_{k \ge N} (1 - |z_k|^2)\delta_{z_k}$$ is a Carleson measure with 
Carleson embedding constant $\CC(\mu_N)$ satisfying $\CC(\mu_N) \to 1$ as $N \to \infty$;
\item The measure $$\nu_N = \sum_{k \ge N}\frac{(1 - |z_k|^2)}{\delta_k} \delta_{z_k}$$ is a Carleson measure with embedding constant $\RR_{\nu_N}$ 
on reproducing kernels satisfying $\RR_{\nu_N} \to 1$.
\end{enumerate}
Moreover, if $\{z_n\}$ is an $EIS_p$ $($strong-$AIS_p$$)$ sequence for some $p$ with $1 \le p \le \infty$, then it is an $EIS_p$ $($strong $AIS_p$$)$ sequence for all $p$.
\end{thm}

In what follows, we let $\tilde{\delta}_N = \min\{\delta_n: n \ge N\}$. Since $\delta_n > \delta > 0$ for all $n$ and $\delta_n \to 1$, we see that $\tilde{\delta}_N \to 1$ as $N \to \infty$.

\begin{proof} {\bf (1) implies (2)}: Suppose that for some $p$ we know that $\{z_n\}$ is an $EIS_p$ sequence. 
Let $B$ denote the corresponding Blaschke product. Suppose $\{z_n\}$ is not thin.  Then there exists $\varepsilon > 0$ and a subsequence $\{z_{n_k}\}$ such that 
$$\sup_k |B_{n_k}(z_{n_k})| \le 1 - \varepsilon.$$ 

First suppose that $p < \infty$. Choose a subsequence of $\{z_{n_k}\}$ recursively, denoted $\{z_{n_{k_m}}\}$, such that
$\{z_{n_{k_m}}\}$ is thin and satisfies
\begin{eqnarray}\label{sumcondition}
(1 - |z_{n_{k_m}}|^2)^{-1/p} \left(  \sum_{j > m} (1 - |z_{n_{k_j}}|^2)\right)^{1/p} \to 0.\end{eqnarray}

Now by our assumption (1) and the fact that $a_k = (1 - |z_k|^2)^{1/p} \in \ell^p$, we know that there exists $N = N(\varepsilon)$ such that for ${k_m} \ge N$ there exists $f_m \in H^p$ such that for $j \ge N$
 
\begin{multline*}
f_m(z_{n_{k_l}}) (1 - |z_{n_{k_l}}|^2)^{1/p} =(1 - |z_{n_{k_l}}|^2)^{1/p} ~\mbox{for}~ l \ge m ~ \\
\mbox{and}~ f_m(z_j) = 0 ~\mbox{for}~ j \notin \{n_{k_s}\}_{k_s \ge N} ~\mbox{or}~ j = {n_{k_l}}, N \le k_l < k_m,
\end{multline*}

and 

$$\|f_m\|_p \le (1 + \varepsilon)\left(\sum_{n \ge N} |a_n|^p\right)^{1/p} = (1 + \varepsilon)\left(\sum_{l \ge m} (1 - |z_{n_{k_l}}|^2)\right)^{1/p}.$$
Let $b_1$ denote the (thin) Blaschke product with zeros $\{z_{n_{k_l}}\}$ for $n_{k_l} > N$, and $b_2$ denote the Blaschke product with zeros $\{z_j\}_{j \notin \{n_{k_l}\}, j\ge N}$. Now for each such $m$, we know that $f_m$ vanishes on the zeros of $b_2$ so write $f_m = b_2 g_m$ for some $g_m \in H^p$. Since $g_m \in H^p$, we know that
$|g_m(z)| \le (1 - |z|^2)^{-1/p}\|g_m\|_p$, so we obtain

\begin{eqnarray*}
1 & = & |f_m(z_{n_{k_m}})| = |b_2(z_{n_{k_m}})| \, |g_m(z_{n_{k_m}})|\\
& \le & |b_2(z_{n_{k_m}})| (1 - |z_{n_{k_m}}|^2)^{-1/p} \|g_m\|_p\\
& = & |b_2(z_{n_{k_m}})| (1 - |z_{n_{k_m}}|^2)^{-1/p} \|f_m\|_p\\
& \le & |b_2(z_{n_{k_m}})|\frac{1}{(1 - |z_{n_{k_m}}|^2)^{1/p}}\left( (1 + \varepsilon) \left( (1 - |z_{n_{k_m}}|^2) + \sum_{j > m} (1 - |z_{n_{k_j}}|^2)\right)^{1/p}\right).
\end{eqnarray*}

Thus,  
$$1 \le (1 + \varepsilon)\liminf_{m \to \infty}  |b_2(z_{n_{k_m}})|\left( 1 + \left(\frac{1}{1 - |z_{n_{k_m}}|^2}\sum_{j > m} (1 - |z_{n_{k_j}}|^2)\right)\right)^{1/p}.$$

Therefore, there exists $\eta_m \to 0$ such that 
$\liminf_{m \to \infty}|b_2(z_{n_{k_m}})| \ge \lim_{m \to \infty} \frac{1}{(1 + \varepsilon)(1 + \eta_{m})}$ and since we assume that  $\sup_l |B_{n_{k_l}}(z_{n_{k_l}})| \le 1 - \varepsilon$ we have
\begin{eqnarray*}
 1 - \varepsilon & \ge & \liminf_{m \to \infty} \prod_{j \le N} \rho(z_{n_{k_m}}, z_j) \, \prod_{j \in \{z_{n_{k_l}}, l \ne m\}, j > N}
 \rho(z_{n_{k_m}}, z_j) \, \prod_{j \notin \{z_{n_{k_m}}\}, j > N} \rho(z_{n_{k_m}}, z_j)\\
 &  =  & \liminf_{m \to \infty} \prod_{j \le N} \rho(z_{n_{k_m}}, z_j) 
 |{b_1}_{z_{n_{k_m}}}(z_{n_{k_m}})| |b_2(z_{n_{k_m}})| \ge \frac{1}{1 + \varepsilon},\end{eqnarray*} a contradiction.  The case $p = \infty$ follows
 from Corollary \ref{cor:eip}.

 \bigskip
{\bf (2) implies (1).} We show (2) implies (1) for all $p$:  We have shown, in Theorem~\ref{thm:replacement}, that a thin interpolating sequence is an $EIS_\infty$ sequence. We check that thin implies that it is an $EIS_p$ sequence for all $p$ with $1 \le p \le \infty$. To see this, let $\varepsilon > 0$ be given. Since we assume $\{z_j\}_{j \ge N}$ is a thin interpolating sequence, we may choose the P. Beurling functions with constant $(1+ \varepsilon)$
given by Theorem~\ref{thm:replacement}. Then consider 

$$G_N(z) = \sum_{n \ge N} a_n f_n(z) (h_{z_n}(z))^{2/p}.$$ 
We see that $G_N(z_j) = a_j (1 - |z_j|^2)^{-1/p}$ when $j\geq N$ and 

\begin{equation}
\label{Estimate}
|G_N(z)| \le \sum_{n \ge N} (1 + \varepsilon)|a_n| |h_{z_n}(z)|^{2/p} \frac{|f_n(z)|}{1 + \varepsilon}.
\end{equation}
Applying H\"older's inequality and raising both sides to the $p$-th power, we get
$$|G_N(z)|^p \le \left(\sum_{n \ge N} (1 + \varepsilon)^p |a_n|^p |h_{z_n}(z)|^2\right)\left(\sum_{n \ge N} \frac{|f_n(z)|^q}{(1 + \varepsilon)^q}\right)^{p/q}.$$
But $q \ge 1$ and $|f_n(z)| < 1 + \varepsilon$ for all $n$ and $z$, so 
\begin{eqnarray*}
|G_N(z)|^p & \le & (1 + \varepsilon)^p \left(\sum_{n \ge N} |a_n|^p |h_{z_n}(z)|^2\right)\left(\sum_{n \ge N} \frac{|f_n(z)|}{1 + \varepsilon}\right)^{p/q}\\
& \le & (1 + \varepsilon)^p \sum_{n \ge N} |a_n|^p |h_{z_n}(z)|^2.\end{eqnarray*}
Integrating, we get $$\|G_N\|_p^p \le (1 + \varepsilon)^p \|a\|_{N, \ell^p}^p.$$ 
The case of $p=1$ follows from \eqref{Estimate} and the fact that $\sum_{n\geq N}\abs{f_n(z)}<1+\varepsilon$ for all $n$ and $z$ and then integration.  Therefore, we see that (2) implies (1) for all $p$. 

\bigskip
{\bf (2) is equivalent to (3)}: This follows from Theorem~\ref{Volberg}.

\bigskip

{\bf (4) and (1) are equivalent}: This follows from Theorem~\ref{EISiffASI}. \\

At this point, we will note that if $\{z_n\}$ is an $EIS_{p_0}$ sequence for some $p_0$, then it is an $EIS_p$ sequence for all $p$, and the same is true for $AIS_p$ sequences: Suppose $\{z_n\}$ is an $EIS_p$ sequence. Then, from our work above, it is a thin interpolating sequence. Therefore it is $EIS_p$ for all $p$. If $AIS_p$ holds for some $p$, we know $EIS_p$ holds for that $p$ and therefore for all $p$. Consequently, $AIS_p$ holds for all $p$, too. Therefore (1) - (4) are equivalent and if (1) or (4) is true for some $p$, they are both true for all $p$.

\bigskip

The remainder of the theorem follows from Theorem \ref{thm:Carleson} once one observes that $\delta_k$ is bounded below.

\end{proof}

A thin sequence allows repetition of finitely many points. It is clear that we may remove finitely many points from such a sequence to obtain an interpolating sequence. Recall that a sequence is an $AIS$ sequence of type $1$ for $H^\infty$ if whenever $a \in \ell^\infty$, there exists $g \in H^\infty$ such that $|g(z_j) - a_j| \to 0$ and $\|g\|_\infty \le \|a\|_{\ell^\infty}$. Although this notion does not appear to be equivalent to that of a strong $AIS$ sequence for $H^p$, it is a consequence of Theorem~\ref{thm:asiinfty} that an $AIS$ of type $1$ for $H^\infty$ is equivalent to being a thin interpolating sequence when the points $\{z_n\}$ are distinct, for in this case an $AIS$ sequence of type $1$ for $H^\infty$ is equivalent to being a thin sequence and that, by Theorem~\ref{thm:asiinfty}, is equivalent to being a strong $AIS_p$ sequence.  We summarize these remarks below.

\begin{cor} A sequence $\{z_n\}$  of distinct points is an $AIS$ sequence of type $1$ for $H^\infty$ if and only if it is a strong $AIS_\infty$ sequence.
\end{cor}

\section*{Acknowledgement}
We thank the anonymous referee for his or her comments.



\end{document}